\documentclass[11pt,a4paper,reqno]{amsart}
\usepackage{amsmath}
\usepackage{amsfonts}
\usepackage{amssymb}
\usepackage{amscd}
\usepackage{url}
\usepackage[pdftex,bookmarks=true]{hyperref}

\newcommand\Vol{{\operatorname{Vol}}}
\newcommand\rank{{\operatorname{rank}}}

\newcommand\R{{\mathbf{R}}}

\renewcommand\P{{\mathbf{P}}}
\newcommand\E{{\mathbf{E}}}

\newcommand\Z{{\mathbf{Z}}}

\newcommand\row{{\mathbf{r}}}

\newcommand\Ba{{\mathbf a}}

\newcommand\Bu{{\mathbf u}}
\newcommand\Bv{{\mathbf v}}
\newcommand\Bw{{\mathbf w}}
\newcommand\Bx{{\mathbf x}}
\newcommand\By{{\mathbf y}}
\newcommand\Bz{{\mathbf z}}
\newcommand\ep{\epsilon}

\theoremstyle{plain}
  \newtheorem{theorem}[subsection]{Theorem}

  \newtheorem{lemma}[subsection]{Lemma}

  \newtheorem{example}[subsection]{Example}
  \newtheorem{condition}{Condition}

\theoremstyle{remark}

\theoremstyle{definition}

\begin{document}

\title[Inverse Littlewood-Offord problem for quadratic forms]{A continuous variant of the inverse Littlewood-Offord problem for quadratic forms}

\author{Hoi H. Nguyen}
\address{Department of Mathematics, University of Pennsylvania, 209 South 33rd Street, Philadelphia, PA 19104, USA}
\email{hoing@math.upenn.edu}

\maketitle

\begin{abstract}
Motivated by the inverse Littlewood-Offord problem for linear forms, we study the concentration of quadratic forms. We show that if this form concentrates on a small ball with high probability, then the coefficients can be approximated by a sum of additive and algebraic structures.   
\end{abstract}

\section{Introduction}\label{section:introduction}
\subsection{The Littlewood-Offord problem for linear forms} 
Let $\xi$ be a real random variable, and let $A=\{a_1,\dots,a_n\}$ be a multiset in $\R^d$. For any $\beta>0$, we define the {\it small ball probability} as

$$\rho_{\beta,\xi}(A):=\sup_{a\in \R^d} \P_{\Bx}\big(a_1x_1+\dots+a_nx_n \in B(a,\beta)\big),$$

where $\Bx=(x_1,\dots, x_n)$ and $x_i$ are iid copies of $\xi$, and $B(x,\beta)$ denotes the closed disk of radius $\beta$ centered at $x$ in $\R^d$. 

A classical result of Erd\H{o}s \cite{E} and Littlewood-Offord \cite{LO} asserts that if $\xi$ has Bernoulli distribution and $a_i$ are real numbers of magnitude $|a_i|\ge \beta$, then 

$$\rho_{\beta,\xi}(A)= O(n^{-1/2}).$$ 

This remarkable inequality has generated an impressive way of research, particularly from the early 1960s to the late 1980s. We refer the reader to \cite{Ess,H,Kle} and the references therein.

Motivated by inverse theorems from additive combinatorics (see \cite[Chapter 5]{TVbook}), Tao and Vu brought a new view to the problem: find the underlying reason as to why the small ball probability is large (say, polynomial in $n$). 

Typical examples of $A$, where $\rho_{\beta,\xi}$ is large, involve \emph{generalized arithmetic progressions} (GAPs), an important concept from additive combinatorics. 

A set $Q\subset \R^d$ is a \emph{GAP of rank $r$} if it can be expressed as in the form
$$Q= \{g_0+ k_1g_1 + \dots +k_r g_r| K_i \le k_i \le K_i' \hbox{ for all } 1 \leq i \leq r\}$$ for some $g_0,\ldots,g_r\in \R^d$, and some integers $K_1,\ldots,K_r,K'_1,\ldots,K'_r$.

\vskip .1in

It is convenient to think of $Q$ as the image of an integer box $B:= \{(x_1, \dots, x_r) \in \Z^r| K_i \le k_i\le K_i' \} $ under the linear map
$$\Phi: (x_1,\dots, x_r) \mapsto g_0+ x_1g_1 + \dots + x_r g_r. $$

The vectors $g_i$ are the \emph{generators } of $Q$, the numbers $K_i'$ and $K_i$ are the \emph{dimensions} of $Q$, and $\Vol(Q) := |B|$ is the \emph{volume} of $Q$. We say that $Q$ is \emph{proper} if this map is one to one, or equivalently if $|Q| = \Vol(Q)$.  For non-proper GAPs, we of course have $|Q| < \Vol(Q)$.
If $g_0=0$ and $-K_i=K_i'$ for all $i\ge 1$, we say that $Q$ is {\it symmetric}.

\vskip .1in

\begin{example}\label{example:linear}
Let $Q= \{\sum_{i=1}^r k_ig_i | -K_i \le k_i \le K_i\}$ be a proper symmetric GAP of rank $r=O(1)$ and size $N=n^{O(1)}$. Assume that $\xi$ has Bernoulli distribution, and for each $a_i$ there exists $q_i\in Q$ such that $\|a_i-q\|_2\le \delta$. 

Then, because the random sum $\sum_i q_ix_i$ takes value in the GAP $nQ:=\{\sum_{i=1}^r k_ig_i| -nK_i \le k_i \le nK_i\}$, and because $|nQ| \le n^r N=n^{O(1)}$, the pigeon-hole principle implies that $\sum_i q_ix_i$ takes some value in $nQ$ with probability $n^{-O(1)}$. Thus we have

\begin{equation}\label{bound2} 
\rho_{n\delta, \xi}(A)   = n^{-O(1)}.
\end{equation}
\end{example}

The above example shows that if $\xi$ has Bernoulli distribution and if $a_i$ are {\it close} to a $GAP$ of rank $O(1)$ and size $n^{O(1)}$, then $A$ has large small ball probability. 

It was shown (rather implicitly) by Tao and Vu in \cite{TVbull,TVinverse,TVcir,TVcomp} that these are essentially the only examples which have large small ball probability. An explicit version was given by Vu and the current author under the following condition.

\begin{condition}[Anti-concentration]\label{condition:space} There exist positive constants $0<c_1<c_2$ and $c_3$ such that
$$\P(c_1\le |\xi-\xi'|\le c_2) \ge c_3,$$
where $\xi'$ is an independent copy of $\xi$.
\end{condition} 

We note that Bernoulli random variables $\eta^{(\mu)}$ (which equal  $\pm 1$ with probability $\mu/2$ and $0$ with probability $1-\mu$), where the parameters $\mu$ are bounded away from 0, are clearly of this type. 

We say that a vector $a$ is $\delta$-close to a set $Q$ if there exists $q\in Q$ such that $\|a-q\|_2\le \delta$.

\begin{theorem}[Inverse Littlewood-Offord theorem for linear forms, \cite{NgV}]\label{theorem:ILOlinear:continuous} Let $0 <\ep < 1$ and $B>0$. Let 
$ \beta >0$ be a parameter that may depend on $n$. Suppose that $\sum_i\|a_i\|_2^2 =1$ and 
 
$$\rho:=\rho_{\beta,\xi}(A) \ge n^{-B},$$ 

where $x_i$ are iid copies of a random variable $\xi$ satisfying Condition \ref{condition:space}. Then, for any number $n'$ between $n^\ep$ and $n$, there exists a proper symmetric GAP $Q=\{\sum_{i=1}^r k_ig_i : |k_i|\le K_i \}$ such that

\begin{itemize}

\item At least $n-n'$ elements of $a_i$ are $\beta$-close to $Q$.

\vskip .1in

\item $Q$ has small rank, $r=O_{B,\ep}(1)$, and small size

$$|Q| \le \max \big(O_{B,\ep}(\frac{\rho^{-1}}{\sqrt{n'}}),1\big).$$

\vskip .1in

\item There is a non-zero integer $p=O_{B,\ep}(\sqrt{n'})$ such that all
 steps $g_i$ of $Q$ have the form  $g_i=(g_{i1},\dots,g_{id})$, where $g_{ij}=\beta\frac{p_{ij}} {p} $ with $p_{ij} \in \Z$ and $|p_{ij}|=O_{B,\ep}(\beta^{-1} \sqrt{n'}).$

\end{itemize}
\end{theorem}

In this and all subsequent theorems, the hidden constants could also depend on $d$ and $c_1,c_2,c_3$ of Condition \ref{condition:space}. We could have written $O_{d,c_1,c_2,c_3}(.)$ everywhere, but these notations are somewhat cumbersome, and this dependence is not our focus, so we omit them. Theorem \ref{theorem:ILOlinear:continuous} was proven in \cite{NgV} with $c_1=1,c_2=2$ and $c_3=1/2$, but the proof there extends to the general case rather automatically.

{\bf Notation.} Let $x_1,\dots,x_n$ be real numbers, and let $a_1,\dots,a_n$ be vectors in $\R^d$. To simplify our presentation, we will denote the sum vector $\sum_{i}a_ix_i$ by $\Ba \cdot \Bx$, or $\Bx \cdot \Ba$, where $\Bx=(x_1,\dots,x_n)$ and $\Ba=(a_1,\dots,a_n)$. For instance, the small ball probability can be expressed as

$$\rho_{\beta,\xi}(A)= \sup_{a}\P_{\Bx}\big(\Bx \cdot \Ba \in B(a,\beta)\big).$$

\subsection{The Littlewood-Offord problem for quadratic forms} Let $\xi$ be a real random variable, and let $A=(a_{ij})$ be an $n\times n$ {\it symmetric} matrix whose entries are vectors of $\R^d$. For any $\beta>0$, we define the {\it quadratic small ball probability} as

$$\rho_{\beta,\xi}(A):= \sup_{a,b_1,\dots,b_n \in \R^d}\P\big(\sum_{i,j} a_{ij}x_ix_j+ \sum_i b_ix_i\in B(a,\beta)\big).$$

where $x_1,\dots,x_n$ are iid copies of $\xi$.

It follows from \cite[Theorem 3.1]{RS} and \cite[Corollary 4.4]{CTV} that if $\xi$ has Bernoulli distribution and if there are $\Theta(n)$ indices $i$ for each of which there are $\Theta(n)$ indices $j$ such that $\|a_{ij}\|_2 \ge \beta$, then the following holds for some explicit constant $c>0$

\begin{equation}\label{eqn:quadratic:forward}
\rho_{\beta,\xi}(A) = O(n^{-c}).
\end{equation}

By using a recent result of Costello \cite{C}, one can improve the right hand side to $O(n^{-1/2+o(1)})$, which is asymptotically tight. 

It seems that one can improve the bound further by imposing new assumptions on $a_{ij}$. However, this is not our goal here. Motivated by the inverse Littewood-Offord problem for linear forms, we would like to find the underlying reason as to why the quadratic small ball probability is large (say, polynomial in $n$). 

In the following examples, $\xi$ has Bernoulli distribution, and for each $a_{ij}$ there exists $q_{ij}$ such that 

$$\|a_{ij}-q_{ij}\|_2\le \delta.$$ 

\vskip .1in

\begin{example}\label{example:quadratic:1} Let $Q$ be a proper symmetric GAP of rank $r=O(1)$ and size $n^{O(1)}$. Assume that the approximated values $q_{ij}$ belong to $Q$. 

Then, because the random sum $\sum_{i,j}q_{ij}x_ix_j$ takes value in the GAP $n^2Q$, and because the size of $n^2Q$ is $n^{O(1)}$, the pigeon-hole principle implies that $\sum_{i,j}q_{ij}x_ix_j$ takes some value in $n^2Q$ with probability $n^{-O(1)}$. Passing back to $a_{ij}$, we obtain

$$\rho_{n^2\delta,\xi}(A) =n^{-O(1)}.$$ 
  
\end{example}

One observes that this example is similar to Example \ref{example:linear}, in which case $q_{ij}$ have additive structure. However, unlike what we in the linear case, there are examples of different nature where the quadratic small ball probability can be large. 

\begin{example}\label{example:quadratic:2} Assume that $q_{ij}$ can be written as $q_{ij}=k_ib_j+k_jb_i$, where $b_i$ are arbitrary in $\R^d$ and $k_i$ are integers bounded by $n^{O(1)}$ such that

$$\P_\Bx(\sum_i k_ix_i= 0)=n^{-O(1)}.$$ 

Then, we have 

$$\P(\sum_{i,j}q_{ij}x_ix_j =\mathbf{0}) =\P(\sum_{i}k_ix_i  \sum_jb_jx_j =\mathbf{0})=n^{-O(1)}.$$

Passing back to $a_{ij}$, we obtain

$$\rho_{n^2\delta,\xi}(A) =n^{-O(1)}.$$ 

\end{example}

Motivated by \ref{example:quadratic:1} and \ref{example:quadratic:2}, we now consider a more complicated example. 

\begin{example}\label{example:quadratic:3}  Assume that $q_{ij}=q_{ij}' +q_{ij}''$, where $q_{ij}'\in Q$, a proper symmetric GAP of rank $O(1)$ and size $n^{O(1)}$, and $q_{ij}''= k_{i1}b_{1j}+k_{j1}b_{1i}+\dots+k_{ir}b_{rj}+k_{jr}b_{ri}$ , where $r=O(1)$, and $b_{1i},\dots, b_{ri}$ are arbitrary in $\R^d$, and $k_{i1},\dots,k_{ir}$ are integers bounded by $n^{O(1)}$ such that 

$$\P_{\Bx}\big(\sum_i k_{i1}x_i=0,\dots,\sum_i k_{ir}x_i=0\big)=n^{-O(1)}.$$ 

Observe that 

$$\sum_{i,j} q_{ij}x_ix_j = \sum_{i,j}q_{i,j}'x_ix_j + (\sum_{i} k_{i1}x_i)(\sum_{j} b_{1j}x_j)+\dots + (\sum_{i} k_{ir}x_i)(\sum_{j} b_{rj}x_j).$$ 

Thus,

$$\sup_{q\in n^2Q}\P_\Bx(\sum_{i,j}q_{ij}x_ix_j=q)=n^{-O(1)}.$$ 

Passing to $a_{ij}$, we obtain 

$$\rho_{n^2\delta,\xi}(A) =n^{-O(1)}.$$ 
\end{example}

In this example, the matrix $(q_{ij})$ is a sum of two unrelated submatrices $(q_{ij}')$ and $(q_{ij}'')$: one has entries belonging to a GAP of rank $O(1)$ and size $n^{O(1)}$, and one has rank $O(1)$.

Our main theorem partially demonstrates that if $\rho_{\beta,\xi}(A)$ is large, then $a_{ij}$ are close to some $q_{ij}$ taking the form of Example \ref{example:quadratic:3}.

We denote by $\row_i(A)$ the row $(a_{i1},\dots, a_{in})$ of $A$. 

\begin{theorem}[Inverse Littlewood-Offord theorem for quadratic forms]\label{theorem:ILOquadratic:continuous}
Let $0 <\ep < 1$ and $B>0$. Let $ \beta >0$ be a parameter that may depend on $n$. Assume that $a_{ij}=a_{ji}$, and 

$$\rho:=\rho_{\beta,\xi}(A)\ge n^{-B}.$$

Then, there exist an integer $k\neq 0, |k|=n^{O_{B,\ep}(1)}$, a set of $r=O(1)$ rows $\row_{i_1}, \dots, \row_{i_r}$ of $A$, and set $I$ of size at least $n-2n^\ep$ such that for each $i\in I$, there exist integers $k_{ii_1},\dots, k_{ii_r}$, all bounded by $n^{O_{B,\ep}(1)}$, such that the following holds.

\begin{equation}\label{eqn:ILOquadratic}
\P_\Bz\Big(\| \Bz \cdot (k\row_i(A) + \sum_{j} k_{ii_j} \row_{i_j}(A))\|_2 \le \beta n^{O_{B,\ep}(1)}\Big)\ge n^{-O_{B,\ep}(1)},
\end{equation}

where $\Bz=(z_1,\dots, z_n)$ and $z_i$ are iid copies of $\eta^{(1/2)} (\xi-\xi')$, where $\eta^{(1/2)}$ is a Bernoulli random variable of parameter $1/2$ which is independent of $\xi$ and $\xi'$.
\end{theorem}

It follows from \eqref{eqn:ILOquadratic} and from Theorem \ref{theorem:ILOlinear:continuous} that for each $i\in I$, most of the entries of $k\row_i(A) + \sum_{j} k_{ii_j} \row_{i_j}(A)$ are $\beta n^{O_{B,\ep}(1)}$-close to a symmetric GAP of rank $O(1)$ and size $n^{O(1)}$. In other words, Theorem \ref{theorem:ILOquadratic:continuous} asserts that, modulo some special linear combinations of $\row_{i_1}(A),\dots, \row_{i_r}(A)$ (where the coefficients are integers bounded by $n^{O(1)}$), most of the components of $\row_i(A)$ are $\beta n^{O(1)}$-close to a symmetric GAP of rank $O(1)$ and size $n^{O(1)}$. 

Theorem \ref{theorem:ILOquadratic:continuous} seems to be useful. It plays a crucial role in our work \cite{Ng-singular} of establishing polynomial bounds for the singular value of random symmetric matrices. We remark that a discrete version of Theorem \ref{theorem:ILOquadratic:continuous} was discussed in an earlier paper \cite{Ng}.

\section{A rank reduction argument and the full rank assumption}\label{appendix:fullrank}
This section, which is independent of its own, provides a technical lemma we will need for later sections. Informally, it says that if we can find a proper symmetric GAP that contains a given set, then we can assume this containment is non-degenerate.   

Assume that $P=\{k_1g_1+\dots+k_rg_r | -K_i\le k_i \le K_i\}$ is a proper symmetric GAP, which contains a set $U=\{u_1,\dots. u_n\}$. 

We consider $P$ together with the map $\Phi: P \rightarrow \R^r$ which maps $k_1g_1+\dots+k_rg_r$ to $(k_1,\dots,k_r)$. Because $P$ is proper, this map is bijective. 

We know that $P$ contains $U$, but we do not know yet that $U$ is non-degenerate in $P$ in the sense that the set $\Phi(U)$ has full rank in $\R^{r}$. In the later case, we say $U$ {\it spans} P.

\begin{theorem}\label{theorem:fullrank}
Assume that $U$ is a subset of a proper symmetric GAP $P$ of size $r$, then there exists a proper symmetric GAP $Q$ that contains $U$ such that the followings hold.

\begin{itemize}
\item $\rank(Q)\le r$ and $|Q|\le O_r(1)|P|$.

\vskip .1in

\item $U$ spans $Q$, that is, $\phi(U)$ has full rank in $\R^{\rank(Q)}$.
\end{itemize}

\end{theorem}

To prove Theorem \ref{theorem:fullrank}, we will rely on the following lemma.

\begin{lemma}[Progressions lie inside proper progressions, \cite{TVbook}]\label{lemma:embeding}
There is an absolute constant $C$ depending in $d$ such that the following holds. Let $P$ be a GAP of rank $r$ in $\R^d$. Then there is a symmetric proper GAP $Q$ of rank at most $r$ containng $P$ and 

$$|Q|\le r^{Cr^3}|P|.$$ 

\end{lemma}

\begin{proof} (of Theorem \ref{theorem:fullrank}) We shall mainly follow \cite[Section 8]{TVsing}.

Suppose that $\Phi(U)$ does not have full rank, then it is contained in a hyperplane of $\R^r$. In other words, there exist integers $\alpha_1,\dots,\alpha_r$ whose common divisor is one and $\alpha_1k_1+\dots + \alpha_r k_r=0$ for all $(k_1,\dots,k_r)\in \Phi(U)$.

Without loss of generality, we assume that $\alpha_r \neq 0$. We select $w$ so that $g_r=\alpha_r w$, and consider $P'$ be the GAP generated by $g_i':=g_i-\alpha_iw$ for $1\le i \le r-1$. The new symmetric GAP $P'$ will continue to contain $U$, because we have 

\begin{align*}
k_1g_1'+\dots +k_{r-1}g_{r-1}' &= k_1g_1+\dots+k_rg_r - w(\alpha_1k_1+\dots+\alpha_rg_r)\\
&= k_1g_1+\dots+k_rg_r
\end{align*}

for all $(k_1,\dots,k_r)\in \Phi(U)$. 

Also, note that the volume of $P'$ is $2^{r-1}K_1\dots K_{r-1}$, which is less than the volume of $P$.

We next use Lemma \ref{lemma:embeding} to guarantee that $P'$ is symmetric and proper without increasing the rank. 
 
Iterate the process if needed. Because the rank of the newly obtained proper symmetric GAP decreases strictly after each step, the process must terminate after at most $r$ steps.
 
\end{proof}

\section{A decoupling lemma and  inverse problem for bilinear forms}

As the first step to establish Theorem \ref{theorem:ILOquadratic:continuous}, we pass to bilinear forms by using a decoupling technique. 

Let $U$ be a subset of $\{1,\dots,n\}$. Let $A_U$ be a symmetric matrix of size $n$ by $n$ defined as

\[
A_U(ij)= 
\begin{cases}
a_{ij} & \text{ if either $i\in U$ and $j\notin  U$, or $i\notin U$ and $j\in U$},\\
0 & \text{otherwise,}
\end{cases}
\]

where we denoted by $A_U(ij)$ the $ij$ entry of $A_U$. 

\begin{lemma}[Decoupling lemma]\label{lemma:decoupling} Assume that 

$$\rho=\sup_{a,b_1,\dots,b_n} \P_{\Bx} \Big(\|\sum_{i,j}a_{ij}x_ix_j+\sum_i b_ix_i -a\|_2\le \beta\Big)\ge n^{-B}.$$

Then,

\begin{equation}\label{eqn:ILObilinear:0}
\P_{\Bv,\Bw}\Big(\|\sum_{i,j} A_U(ij)v_iw_j\|_2=O_B(\beta \sqrt{\log n})\Big) =\Theta(\rho^8),
\end{equation}

where $\Bv=(v_1,\dots,v_n)$, $\Bw=(w_1,\dots,w_n)$, and $v_i,w_j$ are iid copies of  $\xi-\xi'$.
\end{lemma}

We refer the reader to Appendix \ref{appendix:decoupling} for a proof of this lemma.

Lemma \ref{lemma:decoupling} asserts that if $\rho_{\beta,\xi}(A)$ is large then $\sum_{i,j}A_U(ij)v_iw_j$ has small norm with high probability. This fact allows us to deduce useful information for $A_U$ (for all $U$) by combining with the following inverse-type result.

\begin{theorem}[Inverse Littlewood-Offord theorem for bilinear forms]\label{theorem:ILObilinear:continuous}
Let $0 <\ep < 1$ and $B>0$. Let $ \beta >0$ be a parameter that may depend on $n$. Assume that  

$$\sup_{a} \P_{\Bx,\By}(\|\sum_{i,j\le n}a_{ij}x_iy_j -a\|_2\le \beta)\ge n^{-B},$$
 
where $\Bx=(x_1,\dots,x_n), \By=(y_1,\dots,y_n)$, and $x_i$ and $y_i$ are iid copies of a random variable $\xi$ satisfying Condition \ref{condition:space}.
Then, there exist an integer $k\neq 0, |k|=n^{O_{B,\ep}(1)}$, a set of $r=O(1)$ rows $\row_{i_1}, \dots, \row_{i_r}$ of $A$, and set $I$ of size at least $n-2n^\ep$ such that for each $i\in I$, there exist integers $k_{ii_1},\dots, k_{ii_r}$, all bounded by $n^{O_{B,\ep}(1)}$, such that the following holds.

\begin{equation}\label{eqn:special}
\P_\By\Big(\| \By \cdot (k\row_i(A) + \sum_{j} k_{ii_j} \row_{i_j}(A))\|_2 \le \beta n^{O_{B,\ep}(1)}\Big)\ge n^{-O_{B,\ep}(1)}.
\end{equation}

\end{theorem}

For the rest of this section, we prove Theorem \ref{theorem:ILObilinear:continuous}.

First of all, for minor technical reasons, it is convenient to assume $\xi$ to have discrete distribution. The continuous case can be recovered by approximating the continuous distribution by a discrete one while holding $n$ fixed.  

For short, we denote the vector $(a_{i1},\dots,a_{in})$ by $\Ba_i$. We begin by applying Theorem \ref{theorem:ILOlinear:continuous}.

\begin{lemma}\label{lemma:roworthogonal} 
Let $\ep<1$, and $B$ be positive constants. Assume that 

$$\rho=\sup_a\P_{\Bx,\By} \big(|\sum_{i,j}a_{ij}x_iy_j -a|\le \beta\big)\ge n^{-B}.$$ 

Then, the following holds with probability at least $3\rho/4$ with respect to $\By=(y_1,\dots,y_n)$. There exist a proper symmetric GAP $Q_\By\subset \R^d$ of rank $O_{B,\ep}(1)$ and size $\max(O_{B,\ep}(\rho^{-1}/n^{\ep/2}),1)$, and an index set $I_\By$ of size $n-n^\ep$ such that $ \Ba_i \cdot \By$ is $\beta$-close to $Q_\By$ for all $i\in I_{\By}$.
\end{lemma}

\begin{proof}(of Lemma \ref{lemma:roworthogonal}) 
Write

$$\sum_{i,j} a_{ij}x_iy_j = \sum_{i=1}^n x_i (\Ba_i\cdot \By).$$

We say that a vector $\By=(y_1,\dots,y_n)$ is {\it good} if

$$\P_\Bx\big(|\sum_{i=1}^n x_i (\Ba_i \cdot \By) -a|\le \beta\big)\ge \rho/4.$$
 
We call $\By$ {\it bad} otherwise. 

Let $G$ denote the collection of good vectors. We are going to estimate the probability $p$ of a randomly chosen vector $\By=(y_1,\dots,y_n)$ being bad by an averaging method.

\begin{align*}
\P_{\By} \P_{\Bx} \big(|\sum_{i=1}^n x_i(\Ba_i \cdot \By)-a|\le \beta\big) &=\rho\\
p \rho/4 + 1-p &\ge \rho\\
(1-\rho)/(1-\rho/4) &\ge p.
\end{align*}

Thus, the probability of a randomly chosen $\By$ belonging to $G$ is at least 

$$1-p \ge (3\rho/4)/(1-\rho/4) \ge 3\rho/4.$$

Consider a good vector $\By\in G$. By definition, we have

$$\P_{\Bx} \big(|\sum_{i=1}^n x_i (\Ba_i \cdot \By) -a|\le \beta\big) \ge \rho/4.$$

Next, if $\Ba_i \cdot \By =\mathbf{0}$ for all $i$, then the conclusion of the lemma holds trivially for $Q_\By :=\mathbf{0}$. Otherwise, we apply Theorem \ref{theorem:ILOlinear:continuous} to the sequence $\{\Ba_i \cdot \By$, $i=1,\dots,n\}$ (after a rescaling). As a consequence, we obtain an index set $I_\By$ of size $n-n^\ep$ and a proper symmetric GAP $Q_{\By}$ of rank $O_{B,\ep}(1)$ and size $\max(O_{B,\ep}(\rho^{-1}/n^{\ep/2}),1)$, together with its elements $q_i(\By)$, such that $\|\Ba_i \cdot \By  -q_i(\By)\|_2\le \beta$ for all $i\in I_\By$.
\end{proof}

We now work with $q_i(\By)$, where $\By\in G$. 

{\bf Common generating indices}. By Theorem \ref{theorem:fullrank}, we may assume that the $q_i(\By)$ span $Q_{\By}$. We choose from $I_\By$ $s$ indices $i_{y_1},\dots,i_{y_s}$ such that $q_{i_{y_j}}(\By)$ span $Q_\By$, where $s$ is the rank of $Q_\By$. Note that $s=O_{B,\ep}(1)$ for all $\By\in G$. 

Consider the tuples $(i_{y_1},\dots,i_{y_s})$ for all $\By\in G$. Because there are $\sum_{s} O_{B,\ep}(n^s) = n^{O_{B,\ep}(1)}$ possibilities these tuples can take, there exists a tuple, say $(1,\dots,r)$ (by rearranging the rows of $A$ if needed), such that $(i_{y_1},\dots,i_{y_s})=(1,\dots,r)$ for all $\By\in G'$, a subset $G'$ of $G$ which satisfies 

\begin{equation}
\P_\By(\By\in G')\ge \P_\By(\By\in G)/n^{O_{C,\ep}(1)} =\rho/n^{O_{B,\ep}(1)}.
\end{equation}

{\bf Common coefficient tuple}. For each $1\le i\le r$, we express $q_i(\By)$ in terms of the generators of $Q_\By$ for each $\By\in G'$, 

$$q_i(\By) = c_{i1}(\By)g_{1}(\By)+\dots + c_{ir}(\By)g_{r}(\By),$$ 

where $c_{i1}(\By),\dots c_{ir}(\By)$ are integers bounded by $n^{O_{B,\ep}(1)}$, and $g_{i}(\By)$ are the generators of $Q_\By$.

We will show that there are many $\By$ that correspond to the same coefficients $c_{ij}$. 

Consider the collection of the coefficient-tuples $\Big(\big(c_{11}(\By),\dots,c_{1r}(\By)\big);\dots; \big(c_{r1}(\By),\dots c_{rr}(\By)\big)\Big)$ for all $\By\in G'$. Because the number of possibilities these tuples can take is at most

$$(n^{O_{B,\ep}(1)})^{r^2} =n^{O_{B,\ep}(1)}.$$

There exists a coefficient-tuple, say  $\Big((c_{11},\dots,c_{1r}),\dots, (c_{r1},\dots c_{rr})\Big)$, such that  

$$\Big(\big(c_{11}(\By),\dots,c_{1r}(\By)\big);\dots; \big(c_{r1}(\By),\dots c_{rr}(\By)\big)\Big) =\Big((c_{11},\dots,c_{1r}),\dots, (c_{r1},\dots c_{rr})\Big)$$  

for all $\By\in G''$, a subset of $G'$ which satisfies 

\begin{equation}
\P_\By(\By\in G'')\ge \P_\By(\By\in G')/n^{O_{B,\ep}(1)} \ge \rho/n^{O_{B,\ep}(1)}.
\end{equation}

In summary, there exist $r$ tuples  $(c_{11},\dots,c_{1r}),\dots, (c_{r1},\dots c_{rr})$, whose components are integers bounded by $n^{O_{B,\ep}(1)}$, such that the followings hold for all $\By\in G''$.

\begin{itemize}

\item $q_i(\By) = c_{i1}g_{1}(\By)+\dots + c_{jr}g_{r}(\By)$, for $i=1,\dots,r$.

\vskip .1in

\item The vectors  $(c_{11},\dots,c_{1r}),\dots, (c_{r1},\dots c_{rr})$ span $\Z^{\rank(Q_\By)}$.

\end{itemize} 

Next, because $|I_\By|\ge n-n^\ep$ for each $\By\in G''$, by an averaging argument, there exists a set $I$ of size $n-2n^\ep$ such that for each $i\in I$ we have 

\begin{equation}
\P_\By(i\in I_\By, \By\in G'') \ge \P_\By(\By\in G'')/2.
\end{equation}  

From now on we fix an arbitrary row $\Ba$ of index from $I$. We will focus on those $\By\in G''$ where the index of $\Ba$ belongs to $I_\By$. 

{\bf Common coefficient tuple for each individual.} Because $q(\By) \in Q_{\By}$ ($q(\By)$ is the element of $Q_\By$ that is $\beta$-close to $\Ba \cdot \By$), we can write 

$$q(\By) = c_{1}(\By)g_{1}(\By)+\dots c_{r}(\By)g_{r}(\By)$$ 

where $c_{i}(\By)$ are integers bounded by $n^{O_{B,\ep}(1)}$.

For short, for each $i$ we denote by $\Bv_i$ the vector $(c_{i1},\dots,c_{ir})$, we will also denote by $\Bv_{\Ba,\By}$ the vector $(c_{1}(\By),\dots c_{r}(\By))$. 

Because $Q_{\By}$ is spanned by $q_1(\By),\dots, q_{r}(\By)$, we have $k=\det(\mathbf{v}_1,\dots \mathbf{v}_r)\neq 0$, and that 

\begin{equation}\label{eqn:ILObilinear:det}
k q(\By) + \det(\mathbf{v}_{\Ba,\By},\mathbf{v}_2,\dots,\mathbf{v}_r)q_{1}(\By) +\dots + \det(\mathbf{v}_{\Ba,\By},\mathbf{v}_1,\dots,\mathbf{v}_{r-1})q_{r}(\By) =0.
\end{equation}

It is crucial to note that $k$ is independent of the choice of $\Ba$ and $\By$. 

Next, because each coefficient of \eqref{eqn:ILObilinear:det} is bounded by $n^{O_{B,\ep}(1)}$, there exists a subset $G_{\Ba}''$ of $G''$ such that all $\By\in G_{\Ba}''$ correspond to the same identity, and

\begin{equation}
\P_\By(\By\in G_{\Ba}'') \ge (\P_\By(\By\in G'')/2)/(n^{O_{B,\ep}(1)})^r = \rho/n^{O_{B,\ep}(1)} = n^{-O_{B,\ep}(1)}.
\end{equation}

In other words, there exist integers $k_1,\dots,k_r$ depending on $\Ba$, all bounded by $n^{O_{B,\ep}(1)}$, such that 

\begin{equation}\label{eqn:ILObilinear:q}
k q(\By) + k_1 q_{1}(\By) + \dots + k_r q_{r}(\By)=0
\end{equation}

for all $\By\in G_{\Ba}''$. 

{\bf Passing back to $A$.} Because $q_i(\By)$ are $\beta$-close to $\Ba_i \cdot \By$, it follows from \eqref{eqn:ILObilinear:q} that

\begin{equation}\label{eqn:ILObilinear:a}
\|k \Ba \cdot \By + k_1  \Ba_1 \cdot \By + \dots + k_r \Ba_{r} \cdot \By\|_2 = \|(k\Ba+k_1\Ba_1+\dots+\Ba_r)\cdot \By \|_2\le n^{O_{B,\ep}(1)}\beta.
\end{equation}

Furthermore, as $\P_\By(\By\in G_{\Ba}'') =n^{-O_{B,\ep}(1)}$, we have

\begin{equation}\label{eqn:ILObilinear:b}
\P_{\By}(\|(k\Ba+k_1\Ba_1+\dots+k_r\Ba_r)\cdot \By \|_2\le n^{O_{B,\ep}(1)}\beta)=n^{-O_{B,\ep}(1)}.
\end{equation}

Because \eqref{eqn:ILObilinear:b} holds for any row $\Ba$ indexing from $I$, we have obtained the conclusion of Theorem \ref{theorem:ILObilinear:continuous}.

\section{proof of Theorem \ref{theorem:ILOquadratic:continuous}}\label{section:ILOquadratic:proof}

By the definition of $\xi$, it is clear that the random variable $\xi-\xi'$ also satisfies Condition \ref{condition:space} (with different positive parameters). We next apply Theorem \ref{theorem:ILObilinear:continuous} to \eqref{eqn:ILObilinear:0} to obtain the following lemma. 

\begin{lemma}\label{lemma:quadratic:row} There exist a set $I_0(U)$ of size $O_{B,\ep}(1)$ and a set $I(U)$ of size at least $n-n^\ep$, and a nonzero integer $k(U)$ bounded by $n^{O_{B,\ep}(1)}$ such that for any $i\in I$, there are integers $k_{ii_0}(U), i_0\in I_0(U)$, all bounded by $n^{O_{B,\ep}(1)}$, such that 

$$\P_\By \Big(\|(k(U)\Ba_i(A_U)+ \sum_{i_0\in I_0} k_{ii_0}(U) \Ba_{i_0}(A_U))\cdot \By\|_2 \le \beta n^{O_{B,\ep}(1)}\Big) = n^{-O_{B,\ep}(1)},$$

where $\By=(y_1,\dots,y_n)$ and $y_i$ are iid copies of $\xi-\xi'$.
\end{lemma}

Note that this lemma holds for all $U\subset [n]$. In what follows we will gather these information. 

As $I_0(U)\subset [n]^{O_{B,\ep}(1)}$ and $k(U)\le n$, there are only $n^{O_{B,\ep}(1)}$ possibilities that the tuple $(I_0(U),k(U))$ can take. Thus, there exists a tuple $(I_0,k)$ such that 
$I_0(U)=I_0$ and $k(U)=k$ for $2^n/n^{O_{B,\ep}(1)}$ different sets $U$. Let us denote this set of $U$ by $\mathcal{U}$; we have

$$|\mathcal{U}|\ge 2^n/n^{O_{B,\ep}(1)}.$$

Next, let $I$ be the collection of all $i$ which belong to at least $|\mathcal{U}|/2$ index sets $I_U$. Then,
 
\begin{align*}                         
|I||\mathcal{U}| + (n-|I|)|\mathcal{U}|/2 & \ge (n-n^\ep )|\mathcal{U}|\\
|I| &\ge  n-2n^\ep.
\end{align*}

From now on we fix an $i\in I$. Consider the tuples $(k_{ii_0}(U), i_0\in I_0)$ over all $U$ where $i\in I_U$. Because there are only $n^{O_{B,\ep}(1)}$ possibilities such tuples can take, there must be a tuple, say $(k_{ii_0}, i_0\in I_0)$, such that $(k_{ii_0}(U), i_0\in I_0)=(k_{ii_0}, i_0\in I_0)$ for at least $|\mathcal{U}|/2n^{O_{B,\ep}(1)}=2^n/n^{O_{B,\ep}(1)}$ sets $U$. 

Because $|I_0|=O_{B,\ep}(1)$, there is a way to partition $I_0$ into $I_0' \cup I_0''$ such that there are $2^n/n^{O_{B,\ep}(1)}$ sets among the $U$ above that satisfy $U\cap I_0=I_0''$. Let $\mathcal{U}_{I_0',I_0''}$ denote the collection of these $U$.

By passing to consider a subset of  $\mathcal{U}_{I_0',I_0''}$ if needed, we may assume that either $i\notin U$ or $i\in U$ for all $U\in  \mathcal{U}_{I_0',I_0''}$. Without loss of generality, we assume the first case. (The other case can be treated similarly).

Let $U\in \mathcal{U}_{I_0',I_0''}$ and $\Bu=(u_1,\dots,u_n)$ be its characteristic vector ($u_j=1$ if $j\in U$, and $u_j=0$ otherwise). 

By the definition of $A_U$, and because $I_0'\cap U=\emptyset$ and $I_0''\subset U$, for any $i_0'\in I_0'$ and $i_0''\in I_0''$ we can write 

$$\Ba_{i_0'}(A_U) \cdot \By  = \sum_{j=1}^n a_{i_0'j}u_jy_j, \mbox{ and } \Ba_{i_0''}(A_U) \cdot \By = \sum_{j=1}^n a_{i_0''j}(1-u_j)y_j.$$

Also, because $i\notin U$, we have

$$\Ba_{i}(A_U)\cdot \By = \sum_{j=1}^n a_{ij}u_jy_j.$$

Thus, 

\begin{align*}
&\quad k\Ba_i(A_U) \cdot \By + \sum_{i_0\in I_0} k_{ii_0} \Ba_{i_0}(A_U) \cdot \By \\
& =  k\Ba_i(A_U) \cdot \By +  \sum_{i_0'\in I_0'} k_{ii_0'} \Ba_{i_0'}(A_U) \cdot \By + \sum_{i_0''\in I_0''} k_{ii_0''} \Ba_{i_0''}(A_U)\cdot \By \\ 
&= \sum_{j=1}^n ka_{ij} u_jy_j + \sum_{j=1}^n \sum_{i_0'\in I_0'} k_{ii_0'} a_{i_0'j} u_jy_j +  \sum_{j=1}^n \sum_{i_0''\in I_0''} k_{ii_0''} a_{i_0''j} (1-u_j)y_j\\
&= \sum_{j=1}^n (ka_{ij} + \sum_{i_0'\in I_0'} k_{ii_0'} a_{i_0'j}- \sum_{i_0''\in I_0''} k_{ii_0''} a_{i_0''j} ) u_jy_j +  \sum_{j=1}^n \sum_{i_0''\in I_0''} k_{ii_0''} a_{i_0''j} y_j.
\end{align*}

Next, by Lemma \ref{lemma:quadratic:row}, the following holds for each $U\in \mathcal{U}_{I_0',I_0''}$

$$\P_\By\Big (\|k \Ba_i(A_U)\cdot \By  + \sum_{i_0\in I_0} k_{ii_0} \Ba_{i_0}(A_U) \cdot \By \|_2 = O(\beta n^{O_{B,\ep}(1)})\Big)=n^{-O_{B,\ep}(1)}.$$ 

Also, recall that 

$$|\mathcal{U}_{I_0',I_0''}|= 2^n/n^{O_{B,\ep}(1)}.$$ 

Hence, 

$$\E_\By\E_U \Big(\|k\Ba_i(A_U) \cdot \By  + \sum_{i_0\in I_0} k_{ii_0} \Ba_{i_0}(A_U) \cdot \By \|_2 =O(\beta n^{O_{B,\ep}(1)})\Big) \ge n^{-O_{B,\ep}(1)}.$$

By applying the Cauchy-Schwarz inequality, we obtain 

\begin{align}\label{eqn:z}
&n^{-O_{B,\ep}(1)}\le \Big[\E_\By \E_U(\|k\Ba_i(A_U)\cdot\By + \sum_{i_0\in I_0} k_{ii_0} \Ba_{i_0}(A_U)\cdot \By \|_2 =O(\beta n^{O_{B,\ep}(1)}))\Big]^2 \nonumber \\
&\le  \E_\By \Big[\E_U(\|k \Ba_i(A_U) \cdot \By  + \sum_{i_0\in I_0}  k_{ii_0} \Ba_{i_0}(A_U)\cdot \By \|_2 =O(\beta n^{O_{B,\ep}(1)}))\Big]^2 \nonumber \\
&= \E_\By \Big[\E_{\Bu}(\|\sum_{j=1}^n (ka_{ij}+\sum_{i_0'\in I_0'} k_{ii_0'} a_{i_0'j}-\sum_{i_0''\in I_0''} k_{ii_0''} a_{i_0''j}) u_jy_j+  \sum_{j=1}^n \sum_{i_0''\in I_0''} k_{ii_0''} a_{i_0''j} y_j\|_2= O(\beta n^{O_{B,\ep}(1)}))\Big]^2 \nonumber \\
&\le \E_\By \E_{\Bu,\Bu'}\Big(\|\sum_{j=1}^n (k_{ij}a_{ij}+\sum_{i_0'\in I_0'} k_{ii_0'} a_{i_0'j}-\sum_{i_0''\in I_0''} k_{ii_0''} a_{i_0''j}) (u_j-u_j')y_j\|_2= O(\beta n^{O_{B,\ep}(1)})\Big) \nonumber \\
&= \E_\Bz\Big(\|\sum_{j=1}^n (ka_{ij}+\sum_{i_0'\in I_0'} k_{ii_0'} a_{i_0'j}-\sum_{i_0''\in I_0''} k_{ii_0''} a_{i_0''j})z_j\|_2 =O(\beta n^{O_{B,\ep}(1)})\Big), 
\end{align}

where  $z_j:=(u_j-u_j')y_j$, and in the last inequality we used the fact that 

$$\E_{\Bu,\Bu'}\Big(\|f(\Bu)\|_2=O(\beta n^{O_{B,\ep}(1)}),\|f(\Bu')\|_2=O(\beta n^{O_{B,\ep}(1)})\Big) \le \E_{\Bu,\Bu'}\Big(\|f(\Bu)-f(\Bu')\|_2=O(\beta n^{O_{B,\ep}(1)})\Big).$$

Note that $u_j-u_j'$ are iid copies of the Bernoulli random variable $2\eta^{(1/2)}$. Hence $z_j$ are iid copies of $2\eta^{(1/2)}(\xi-\xi')$, where $\eta^{(1/2)}$ is independent of $\xi$ and $\xi'$.

In conclusion, the following holds for any $i\in I$,

$$\P_\Bz\Big(\|\sum_{j=1}^n (ka_{ij}+\sum_{i_0'\in I_0'} k_{ii_0'} a_{i_0'j}-\sum_{i_0''\in I_0''} k_{ii_0''} a_{i_0''j})z_j\|_2 =O(\beta n^{O_{B,\ep}(1)})\Big) \ge n^{-O_{B,\ep}(1)}.$$

Note that $k$ and $I_0$ are independent of the choice of $i$. By changing the sign of $k_{ii_0''}$, we are done with the proof of Theorem \ref{theorem:ILOquadratic:continuous}.

\appendix

\section{Proof of Lemma \ref{lemma:decoupling}}\label{appendix:decoupling}
The goal of this section is to establish the inequality

$$\P_{\Bv,\Bw}\Big(\|\sum_{i,j} A_U(ij)v_iw_j\|_2=O_{B}(\beta \sqrt{\log n})\Big) \ge \frac{1}{2}\rho^8/((2\pi)^{7d/2}\exp(8\pi)),$$

under the assumption 

$$\sup_{a,b_1,\dots,b_n} \P_{\Bx} \Big(\|\sum_{i,j}a_{ij}x_ix_j+\sum_ib_ix_i-a|\le \beta\Big)=\rho \ge n^{-B}.$$

Set $a_{ij}':=a_{ij}/\beta$. We have

\begin{align*}
\sup_{a',b_i'} \P_{\Bx} \Big(\|\sum_{i,j}a_{ij}'x_ix_j+\sum_i b_i'x_i-a'\|_2\le 1\Big) \ge n^{-B}.
\end{align*}

Next, by Markov's inequality

\begin{align*}
\P_\Bx\Big(\|\sum_{i,j}a_{ij}'x_ix_j+\sum_i b_i'x_i-a'\|_2\le 1\Big) &= \P\Big(\exp(-\frac{\pi}{2}\|\sum_{i,j}a_{ij}'x_ix_j+\sum_i a_i'x_i-a'\|_2^2 \ge \exp(-\frac{\pi}{2} )\Big)\\
& \le \exp(\frac{\pi}{2} ) \E_\Bx \exp\Big(-\frac{\pi}{2}\|\sum_{i,j}a_{ij}'x_ix_j+\sum_ib_i'x_i-a'\|_2^2\Big).
\end{align*} 

Note that 

$$\exp(-\frac{\pi}{2} \|x\|_2^2) =\int_{\R^d} e(x \cdot t) \exp(-\frac{\pi}{2} \|t\|_2^2) dt.$$

 Thus

$$\P_\Bx\Big(\|\sum_{i,j}a_{ij}'x_ix_j+\sum_i b_i'x_i-a'\|_2\le 1\Big) \le \exp(\frac{\pi}{2}) \int_{\R^d} \Big|\E_\Bx e[(\sum_{i,j}a_{ij}'x_ix_j+\sum_ib_i'x_i)\cdot t]\Big| \exp(-\frac{\pi}{2} \|t\|^2)dt$$
$$\le \exp(\frac{\pi}{2})(\sqrt{2\pi})^d \int_{\R^d} \Big|\E_\Bx e[(\sum_{i,j}a_{ij}'x_ix_j+\sum_ib_i'x_i))\cdot t]\Big| \exp(-\frac{\pi}{2} \|t\|_2^2)/(\sqrt{2\pi})^d dt.$$

Consider $\Bx$ as $(\Bx_U,\Bx_{\bar{U}})$, where $\Bx_U, \Bx_{\bar{U}}$ are the vectors corresponding to $i\in U$ and  $i\notin U$ respectively. By the Cauchy-Schwarz inequality we have

\begin{align*}  
&\quad\Big[\int_{\R^d} \big|\E_\Bx e((\sum_{i,j}a_{ij}'x_ix_j+\sum_i b_i'x_i)\cdot t) \big| \exp(-\frac{\pi}{2} \|t\|_2^2)/(\sqrt{2\pi})^d dt\Big]^4 \\
&\le \Big[\int_{\R^d}  \big|\E_\Bx e((\sum_{i,j}a_{ij}'x_ix_j+\sum_ib_i'x_i))\cdot t)\big|^2 \exp(-\frac{\pi}{2} \|t\|_2^2)/(\sqrt{2\pi})^d dt \Big]^2\\
&\le \Big[\int_{\R^d} \E_{\Bx_U}\big|\E_{\Bx_{\bar{U}}}e((\sum_{i,j} a_{ij}'x_ix_j+\sum_ib_i'x_i))\cdot t)\big|^2 \exp(-\frac{\pi}{2} \|t\|_2^2)/(\sqrt{2\pi})^d dt\Big]^2 \\ 
&=\Big[ \int_{\R^d} \E_{\Bx_U}\E_{\Bx_{\bar{U}},\Bx_{\bar{U}}'} e\Big(\big(\sum_{i\in U,j\in \bar{U}}a_{ij}'x_i(x_j-x_j')+\sum_{j\in \bar{U}} b_j'(x_j-x_j')  \\
&+\sum_{i\in \bar{U},j\in \bar{U}}a_{ij}'(x_ix_j-x_i'x_j')\big)\cdot t\Big) \exp(-\frac{\pi}{2} \|t\|_2^2)/(\sqrt{2\pi})^d  dt\Big]^2\\
&\le \int_{\R^d} \E_{\Bx_{\bar{U}},\Bx_{\bar{U}}'}\Big|\E_{\Bx_{U}} e\Big(\big(\sum_{i\in U,j\in \bar{U}}a_{ij}'x_i(x_j-x_j')+\sum_{j\in \bar{U}} b_j'(x_j-x_j')  \\
&+\sum_{i\in \bar{U},j\in \bar{U}}a_{ij}'(x_ix_j-x_i'x_j')\big)\cdot t\Big)
\Big|^2 \exp(-\frac{\pi}{2} \|t\|_2^2)/(\sqrt{2\pi})^d dt\\
&=\int_{\R^d} \E_{\Bx_U,\Bx_U',\Bx_{\bar{U}},\Bx_{\bar{U}}'} e(\big(\sum_{i\in U, j\in \bar{U}}a_{ij}'(x_i-x_i')(x_j-x_j')\big)\cdot t\Big) \exp(-\frac{\pi}{2} \|t\|_2^2)/(\sqrt{2\pi})^d dt\\
&=\int_{\R^d}\E_{\By_{U},\Bz_{\bar{U}}}e\Big((\sum_{i\in \bar{U},j\in U} a_{ij}'y_iz_j)t\Big)\exp(-\frac{\pi}{2} \|t\|_2^2)/(\sqrt{2\pi})^d dt,
\end{align*}

where $\By_{U}=\Bx_{U}-\Bx_{U}'$ and $\Bz_{\bar{U}}=\Bx_{\bar{U}}-\Bx_{\bar{U}}'$, whose entries are iid copies of $\xi-\xi'$.

Thus we have 

\begin{align*}
&\Big[\int_{\R^d} \big|\E_\Bx e((\sum_{i,j}a_{ij}'x_ix_j)\cdot t)\big|(\exp(-\frac{\pi}{2} \|t\|_2^2)/(\sqrt{2\pi})^d dt\Big]^8\\
&\le \Big[\int_{\R^d}\E_{\By_U,\Bz_{\bar{U}}}e\big((\sum_{i\in U,j\in \bar{U}} a_{ij}'y_iz_j)\cdot t\big) (\exp(-\frac{\pi}{2} \|t\|_2^2)/(\sqrt{2\pi})^d dt\Big]^2\\
&\le \int_{\R^d}\E_{\By_U,\Bz_{\bar{U}}, \By_U', \Bz_{\bar{U}}'}e\Big((\sum_{i\in U,j\in \bar{U}}a_{ij}'y_iz_j-\sum_{i\in U, j\in \bar{U}} a_{ij}'y_i' z_j')\cdot t\Big)\exp(-\frac{\pi}{2} \|t\|_2^2)/(\sqrt{2\pi})^d dt.
\end{align*}

Because $a_{ij}'=a_{ji}'$, we can write the last term as

\begin{align*}
&\int_{\R^d}\E_{\By_U,\Bz_{\bar{U}}', \By_U', \Bz_{\bar{U}}}e\Big(\big(\sum_{i\in U,j\in \bar{U}}a_{ij}'y_iz_j+\sum_{j\in \bar{U}, i\in U} a_{ji}(-z_j')y_i'\big)\cdot t\Big)\exp(-\frac{\pi}{2} \|t\|_2^2)/(\sqrt{2\pi})^d dt\\
&=\int_{\R^d} \E_{\Bv,\Bw}e\Big(( \sum_{i\in U, j\in \bar{U}}a_{ij}'v_iw_j + \sum_{i\in \bar{U}, j\in U} a_{ij}'v_iw_j)\cdot t\Big)\exp(-\frac{\pi}{2} \|t\|_2^2)/(\sqrt{2\pi})^d dt,
\end{align*}

where $\Bv:=(\By_U,-\Bz_{\bar{U}}')$ and $\Bw:=(\By_U', \Bz_{\bar{U}})$. 

Next, reacall that $A_U(ij)=a_{ij}$ if either $i\in U,j\notin U$ or $i\notin U, j\in U$, we have

$$\int_{\R^d} \E_{\Bv,\Bw}e\Big(\big( \sum_{i\in U j\in \bar{U}}a_{ij}'v_iw_j + \sum_{i\in \bar{U}, j\in U} a_{ij}'v_iw_j\big)t\Big)\exp(-\frac{\pi}{2} \|t\|_2^2)/(\sqrt{2\pi})^d dt$$ 
$$= (1/\sqrt{2\pi})^d \E_{\Bv,\Bw}\exp(-\frac{\pi}{2}\|\sum_{i,j}A_U(ij)'v_iw_j\|_2^2),$$

where $A_U(ij)':=A_U(ij)/\beta$.

Thus 

\begin{align*}
\rho^8 &= \Big(\P_\Bx(|\sum_{i,j}a_{ij}'x_i,x_j+\sum_i b_i'x_i -a'|\le 1)\Big)^8 \\
&\le \exp(4\pi)(2\pi)^{4d} \Big(\int_{\R^d} |\E_\Bx e((\sum_{i,j}a_{ij}'x_ix_j)\cdot t)| (\exp(-\frac{\pi}{2} \|t\|_2^2)/(\sqrt{2\pi})^d dt \Big)^8\\
&\le \exp(4\pi)(2\pi)^{7d/2} \E_{\Bv,\Bw}\exp(-\frac{\pi}{2}\|\sum_{i,j}A_U(ij)'v_iw_j\|_2^2).
\end{align*}

Because $\rho\ge n^{-B}$, the inequality above implies that 

$$\P_{\Bv,\Bw}\Big(\|\sum_{i,j} A_U(ij)'v_iw_j\|_2=O_{B}(\sqrt{\log n})\Big) \ge \frac{1}{2} \rho^8/((2\pi)^{7d/2}\exp(4\pi)).$$

Scaling back to $A_{ij}$, we obtain

$$\P_{\Bv,\Bw}\Big(\|\sum_{i,j} A_U(ij)v_iw_j\|_2=O_{B}(\beta \sqrt{\log n})\Big) \ge \frac{1}{2}\rho^8/((2\pi)^{7d/2}\exp(4\pi)),$$

completing the proof.


\end{document}